\documentclass[12pt,a4paper]{article}
\usepackage{latexsym,amssymb,amsfonts,amsmath,amsthm,nccmath,float,enumitem,setspace,bm,authblk}
\usepackage[usenames,dvipsnames]{xcolor}
\usepackage[hidelinks]{hyperref}
\usepackage[margin=2cm]{geometry}
\usepackage[boxed,vlined,longend]{algorithm2e}

\setlist{topsep=3pt,partopsep=0pt,itemsep=1pt,parsep=0pt}

\hypersetup{
	colorlinks,
	linkcolor={red!80!black},
	citecolor={blue!80!black},
	urlcolor={blue!80!black}
}

\newtheorem{Theorem}{Theorem}[section]

\newtheorem{Lemma}{Lemma}[section]
\newtheorem{Example}{Example}[section]
\newtheorem{Claim}{Claim}[section]
\newtheorem{Definition}{Definition}[section]

\newtheorem{Notion }{Notion }[section]

\newcommand{\gaussm}[3]{\genfrac{[}{]}{0pt}{}{#1}{#2}_{#3}}

\newcommand{\A}{\mathcal{A}}
\newcommand{\B}{\mathcal{B}}
\newcommand{\N}{\mathcal{N}}

\newcommand{\f}{\mathbb{F}}

\newcommand{\pg}{{\operatorname{PG}}}

\newcommand{\ipg}{{\operatorname{IPG}}}

\newcommand{\td}{{\operatorname{TD}}}
\newcommand{\e}{{\operatorname{E}}}

\newcommand{\vn}{{\operatorname{V}}}

\newcommand{\sd}{{\operatorname{spread}}}

\def \leq {\leqslant}
\def \geq {\geqslant}

\let\oldproofname=\proofname
\renewcommand{\proofname}{\rm\bf{\oldproofname}}

\begin{document}
	
\title{The chromatic index of finite projective spaces}
	
\author[ ]{Lei Xu}
\author[ ]{Tao Feng}
\affil[ ]{School of Mathematics and Statistics, Beijing Jiaotong University, Beijing 100044, P. R. China}
\affil[ ]{xuxinlei@bjtu.edu.cn, tfeng@bjtu.edu.cn}
	
\renewcommand*{\Affilfont}{\small\it}
\renewcommand\Authands{ and }
\date{}
	
\maketitle
\footnotetext{Supported by NSFC under Grants 12271023 and 12171028}

\begin{abstract}
A line coloring of $\pg(n,q)$, the $n$-dimensional projective space over GF$(q)$, is an assignment of colors to all lines of $\pg(n,q)$ so that any two lines with the same color do not intersect. The chromatic index of $\pg(n,q)$, denoted by $\chi'(\pg(n,q))$, is the least number of colors for which a coloring of $\pg(n,q)$ exists. This paper translates the problem of determining the chromatic index of $\pg(n,q)$ to the problem of examining the existences of $\pg(3,q)$ and $\pg(4,q)$ with certain properties. In particular, it is shown that for any odd integer $n$ and $q\in\{3,4,8,16\}$, $\chi'(\pg(n,q))=(q^n-1)/(q-1)$, which implies the existence of a parallelism of $\pg(n,q)$ for any odd integer $n$ and $q\in\{3,4,8,16\}$.
\end{abstract}
	
	%%%%%%%%%%%%%%%%%%%%%%%%%%%%%%%%%%%%%%%%%%%%%%%%%%%%%%%%%%%%%%%%%%%%%%%%%%%%%%%%%%
	%%%%%%%%%%%%%%%%%%%%%%%%%%%%%%%%%%%%%%%%%%%%%%%%%%%%%%%%%%%%%%%%%%%%%%%%%%%%%%%%%%
	

\section{Introduction}

A {\em Steiner system $S(2,k,v)$} is a pair $\mathcal{D}=(V,\B)$ where $V$ is a set of $v$ {\em points}, and $\B$ is a collection of $k$-subsets of $V$ called {\em blocks}, such that every pair of distinct elements of $V$ is contained in exactly one block of $\B$.
A set of pairwise disjoint blocks in $\mathcal{D}$ is called a {\em partial parallel class}. A partial parallel class containing all the points of $\mathcal{D}$ is called a {\em parallel class}. $\mathcal{D}$ is said to be {\em resolvable} if $\B$ can be partitioned into parallel classes. A {\em block coloring} of $\mathcal{D}$ is a mapping $c:\B\rightarrow C$ where $C$ is a set of colors such that blocks with the same color do not pairwise intersect. A set of blocks in $\mathcal{D}$ with the same color is called a {\em color class} and each color class is a partial parallel class. The {\em chromatic index} of $\mathcal{D}$, denoted by $\chi'(\mathcal{D})$, refers to the minimum number of colors needed to color all blocks. Since the number of blocks contained in each color class is no more than $\lfloor v/k\rfloor$,
$\chi'(\mathcal{D})\geq{|\B|}/{\lfloor v/k \rfloor}$ (cf. \cite[Section 3.2]{rc}).
$\mathcal{D}$ is resolvable if and only if $\chi'(\mathcal{D})=k|\B|/v=(v-1)/(k-1)$.

A special case of a well-known conjecture of Erd\" os, Faber and Lov$\acute{\rm a}$sz \cite{e81} can be formulated as the statement that the chromatic index of any $S(2,k,v)$ is no more than $v$. This conjecture has been shown to hold for any cyclic $S(2,k,v)$ by Colbourn and Colbourn \cite{3uni}, and for any $S(2,k,v)$ with sufficiently large $v$ by Kang, Kelly, K\"uhn, Methuku and Osthus \cite{proof}.
% Besides resolvable designs, there are several sporadic results \cite{rr12,z91} on computing the chromatic index of non-resolvable designs with block size $4$ and $5$.

Let $q$ be a prime power and $\vn(n+1,q)$ be the $(n+1)$-dimensional vector space over the finite field ${\mathbb F}_q$. PG$(n,q)$ denotes the $n$-dimensional projective space over GF$(q)$. A {\em $d$-subspace} of PG$(n,q)$ corresponds to a $(d+1)$-dimensional subspace of $\vn(n+1,q)$. The 0-subspaces and 1-subspaces of $\pg(n,q)$ are called {\em points} and {\em lines}, respectively.
The points of $\pg(n,q)$ together with the lines of $\pg(n,q)$ as blocks and incidence by natural containment form an $S(2,q+1,\gaussm{n+1}{1}{q})$ (cf. \cite[Proposition 2.16]{bjl99}), where $\gaussm{a}{b}{q}:=\prod_{i=1}^{b} \frac{q^{a-b+i}-1}{q^i-1}$, called the Gaussian binomial coefficient, is the number of distinct $b$-dimensional subspaces of $\vn(a,q)$.
Two subspaces are called {\em disjoint} if their intersection is the trivial subspace. Similar to the concept of partial parallel classes, a set of pairwise disjoint lines in $\pg(n,q)$ is called a {\em partial spread}. A partial spread is said to be a {\em spread} if it is a partition of all points of $\pg(n,q)$. A spread exists if and only if $n$ is odd \cite{a54}. A {\em parallelism} of $\pg(n,q)$ is a partition of all lines into spreads. A parallelism of $\pg(n,q)$ forms a resolvable $S(2,q+1,\gaussm{n+1}{1}{q})$. A {\em line coloring} of $\pg(n,q)$ is a mapping from the set of lines of $\pg(n,q)$ to a set of colors
%to color all lines of $\pg(n,q)$
such that any two intersecting lines have different colors. The {\em chromatic index} of $\pg(n,q)$, denoted by $\chi'(\pg(n,q))$, is the least number of colors for which a coloring of $\pg(n,q)$ exists. A parallelism of $\pg(n,q)$ exists if and only if $\chi'(\pg(n,q))=\gaussm{n}{1}{q}$.

%\begin{Proposition}\label{prop:relation between d-para and chrom}
%A $d$-parallelism of $\pg(n,q)$ exists if and only if $\chi'_d(\pg(n,q))=\frac{\gaussm{n-1}{d-1}{q}\left(\gaussm{n+1}{1}{q}-1\right)}{\gaussm{d+1}{1}{q}-1}
%$ $=\frac{\gaussm{n-1}{d-1}{q}\left(q^n-1\right)}{q^d-1}$.
%\end{Proposition}

$\pg(1,q)$ has only one line, so $\chi'(\pg(1,q))=1$. In $\pg(2,q)$ every two lines intersect, so $\chi'(\pg(2,q))=\gaussm{3}{2}{q}=q^2+q+1$. For $n\geq 3$, Beutelspacher, Jungnickel and Vanstone \cite{bjv} showed that $\chi'(\pg(n,q))\leq 4 \gaussm{n-1}{1}{q}+2q^{n-1}$, which confirms the conjecture of Erd\" os, Faber and Lov$\acute{\rm a}$sz \cite{e81} in the case of certain projective spaces. To give a lower bound for $\chi'(\pg(n,q))$ with $n\geq 3$, let $\mu_q(n)$ denote the maximum size of any partial spread of $\pg(n,q)$. Since there are $l:=\gaussm{n+1}{2}{q}$
lines in $\pg(n,q)$, $\chi'(\pg(n,q)\geq l/\mu_q(n)$. If $n$ is odd, a spread of $\pg(n,q)$ exists, so $\mu_q(n)=\gaussm{n+1}{1}{q}/(q+1)$, which yields $\chi'(\pg(n,q))\geq \gaussm{n}{1}{q}$. If $n$ is even, it is known from
\cite{b75,h72} (also, see \cite{ns} for a generalization) that $\mu_q(n)=(q^{n+1}-q^3+q^2-1)/(q^2-1)$ . Since
$$\left\lceil\frac{(q^{n+1}-1)(q^n-1)}{(q^{n+1}-q^{3}+q^2-1)(q-1)}\right\rceil=
\frac{q^n-1}{q-1}+q+\left\lceil\frac{q^4-q^3+q-q^2}{q^{n+1}-q^3+q^2-1}\right\rceil=\frac{q^n-1}{q-1}+q+1,$$
where the last equality holds because of $n\geq 3$, we have $\chi'(\pg(n,q))\geq \gaussm{n}{1}{q}+q+1$. Therefore, we obtain a lower bound for $\chi'(\pg(n,q))$.

\begin{Lemma}\label{prop:lower bound}
For any integer $n\geq 3$ and any prime power $q$,
$$
    \chi'(\pg(n,q))\geq\left\{
          \begin{array}{ll}
               \gaussm{n}{1}{q}, & \text{if}\ n\ \text{is odd}; \vspace{0.2cm}\\
               \gaussm{n}{1}{q}+q+1, & \text{if}\ n\ \text{is even}.
          \end{array}
       \right.
$$
\end{Lemma}

Denniston \cite{d72} employed the Klein correspondence for $\pg(3,q)$ to demonstrate that there is a parallelism of $\pg(3,q)$ for any prime power $q$. Later, Beutelspacher \cite{b74} gave a different proof for this result and, by using induction on $m$, showed that a parallelism of $\pg(2^m-1,q)$ exists for any positive integer $m$ and any prime power $q$. Baker \cite{b76} investigated the partition problem for planes in the $2m$-dimensional affine geometry over GF$(2)$ and found that there is a parallelism of $\pg(2m-1,2)$ for any positive integer $m$. Such a parallelism was also found in the context
of Preparata codes \cite{bvw}. Wettl \cite{w91} provided an explicit proof for the theorem of Baker by prescribing a collineation group. Since the existence of a parallelism of $\pg(n,q)$ implies $\chi'(\pg(n,q))=\gaussm{n}{1}{q}$, we have the following theorem.

\begin{Theorem}\label{thm:known results}
For any positive integer $m$ and any prime power $q$,
\begin{enumerate}
\item[$(1)$] {\rm \cite{b74}} $\chi'(\pg(2^m-1,q))=\gaussm{2^m-1}{1}{q}$, and
\item[$(2)$] {\rm \cite{b76}} $\chi'(\pg(2m-1,2))=\gaussm{2m-1}{1}{2}$.
\end{enumerate}
\end{Theorem}

When $n$ is even, the exact value of $\chi'(PG(n,2))$ was determined by Meszka \cite{m13}.

\begin{Theorem}\label{thm:known result-Meszka}{\rm \cite{m13}}
For any even integer $n\geq 4$, $\chi'(\pg(n,2))=\gaussm{n}{1}{2}+3$.
\end{Theorem}

This paper generalizes Meszka's work in \cite{m13} to establish a framework to count the exact value of $\chi'(\pg(n,q))$ for any integer $n\geq 5$ and any prime power $q$. A special property of $\pg(3,q)$ is introduced and explored in Section \ref{sec:2}. Making use of this property, we translate the problem of determining the chromatic index of $\pg(n,q)$ to the problem of examining the existences of $\pg(3,q)$ and $\pg(4,q)$ with certain properties in Section \ref{sec:3} (see Theorem \ref{thm:idea}). In particular, we obtain the following result, which implies the existence of a parallelism of $\pg(n,q)$ for any odd integer $n$ and $q\in\{3,4,8,16\}$.

\begin{Theorem}\label{thm:main-1}
For any odd integer $n\geq 3$ and $q\in\{3,4,8,16\}$, $\chi'(\pg(n,q))=\gaussm{n}{1}{q}$.
\end{Theorem}

\section{Property E of $\pg(3,q)$}\label{sec:2}

Meszka \cite[Lemma 2]{m13} pointed out that $\pg(3,2)$ has a special property, which can be used to determine the exact value of $\chi'(\pg(n,2))$ for any $n>3$. In this section, we generalize this property to $\pg(3,q)$ for any prime power $q$. This property plays a significant role in the determination of $\chi'(\pg(n,q))$ for any integer $n\geq 5$ in Theorem \ref{thm:idea}.

\begin{Definition}\label{e}
$\pg(3,q)$ is said to have {\em Property $\e$} if it contains a special $\sd$ $P$ and a set  $\mathcal{S}$ of $\gaussm{4}{1}{q}=(q^4-1)/(q-1)=(q^2+1)(q+1)$ spreads such that: $(1)$ for each line $l$ in $P$, there are exactly $q+1$ spreads in $\mathcal{S}$ containing $l$, $(2)$ for each line $l$ not in $P$, there are exactly $q$ spreads in $\mathcal{S}$ containing $l$, and $(3)$ each $\sd$ in $\mathcal{S}$ cannot contain two lines of $P$ at the same time.
\end{Definition}

To better understand the concept of $\pg(3,q)$ with  Property $\e$, we count the number of lines covered by $\mathcal{S}$ in two different ways. There are $(q^4-1)/(q-1)$ spreads in $\mathcal{S}$ and a spread contains $q^2+1$ lines, so $\mathcal{S}$ covers $(q^2+1)\cdot(q^4-1)/(q-1)$ lines. On the other hand, each line in the spread $P$ occurs in exactly $q+1$ spreads in $\mathcal{S}$ and each line not in $P$ occurs in exactly $q$ spreads in $\mathcal{S}$, so $\mathcal{S}$ covers $(q^2+1)\cdot (q+1)+(\gaussm{4}{2}{q}-(q^2+1))\cdot q$ lines. These two numbers are equal.

We shall present a sufficient condition in Lemma \ref{withE} to examine Property E of $\pg(3,q)$. First we recall a basic fact concerning the lines of $\pg(n,q)$.

Since $\mathbb{F}_{q^{n+1}}$ can be regarded as an $(n+1)$-dimensional vector space over $\mathbb{F}_{q}$, each element in the cyclic group $\mathbb{F}_{q^{n+1}}^*/\mathbb{F}_{q}^*$ of order $\gaussm{n+1}{1}{q}$ can be seen as a point of $\pg(n,q)$. Let $\beta$ be a primitive element of $\mathbb{F}_{q^{n+1}}$ and $\alpha=\beta^{\gaussm{n+1}{1}{q}}$ be a primitive element of $\mathbb{F}_q$. We can identify the set of points in $\pg(n,q)$ with $$V=\{1,\beta,\beta^2,\ldots,\beta^{\gaussm{n+1}{1}{q}-1}\},$$
which is a complete set of distinct representatives for the equivalence classes of $\mathbb{F}_{q^{n+1}}^*/\mathbb{F}_{q}^*$ and is closed under addition and multiplication defined over $\mathbb{F}_{q^{n+1}}^*/\mathbb{F}_{q}^*$. The line through two distinct points $x$ and $y$ of $V$ is $\{x, y, x+y, x+\alpha y, \ldots, x+\alpha^{q-2}y\}$. Let $\sigma: x\mapsto\beta x$ be a permutation on $V$ and let $\langle\sigma\rangle$ be the cyclic group generated by $\sigma$. Then the lines of $\pg(n,q)$ can be partitioned into orbits under the action of $\langle\sigma\rangle$. By the orbit-stabilizer theorem, one can obtain the following result (see \cite[Lemma 1.1]{mmj} for example).

\begin{Lemma}\label{orbit}
The lines in $\pg(n,q)$ under the action of $\langle\sigma\rangle$ can be partitioned into
\begin{enumerate}
\item[$(i)$] $\frac{q^n-1}{q^2-1}$ orbits of length $\gaussm{n+1}{1}{q}$ if $n$ is even, or
\item[$(ii)$] a single orbit of length $\gaussm{n+1}{1}{q}/(q+1)$ and $\frac{q^{n}-q}{q^2-1}$ orbits of length $\gaussm{n+1}{1}{q}$ if $n$ is odd.
\end{enumerate}
\end{Lemma}

By Lemma \ref{orbit}, under the action of $\langle\sigma\rangle$, all lines in $\pg(3,q)$ can be partitioned into a {\em short orbit} of length $q^2+1$ and $q$ {\em full orbits} of length $\gaussm{4}{1}{q}$. The following lemma gives a sufficient condition to examine Property E of $\pg(3,q)$. For convenience, throughout this section, we identify $(V,\cdot)$ with the cyclic group $({\mathbb Z}_{\gaussm{n+1}{1}{q}},+)$, and so $\langle\sigma\rangle$ can be seen as $({\mathbb Z}_{\gaussm{n+1}{1}{q}},+)$.

\begin{Lemma}\label{withE}
$\pg(3,q)$ has Property $\e$ if under the action of $\langle\sigma\rangle$, one can take $q$ lines from every full orbit and one line from the short orbit such that the collection of $q^2+1$ lines form a spread, written as $P_0^0$.
\end{Lemma}

\begin{proof} Let $v=\gaussm{4}{1}{q}$ and $P=\{L_{0},L_{1},\ldots,L_{q^2}\}$ be the short orbit of $\pg(3,q)$ under the action of $\langle\sigma\rangle\cong ({\mathbb Z}_{v},+)$, where $L_i=\{j(q^2+1)+i\pmod{v}:0\leq j\leq q\}$ for $0\leq i\leq q^2$. We take $P$ as the special spread in Definition \ref{e}.
%$P_0^0$ is the spread formed by the $q^2+1$ lines given in the assumption.

Under the action of $({\mathbb Z}_{v},+)$, without loss of generality, we can assume that $L_0\in P_0^0$. For $0\leq i\leq q^2$ and $0\leq j\leq q$, let
\begin{align}\notag
P_i^j=P_0^0+j(q^2+1)+i.
\end{align}
Note that $j(q^2+1)+i$ runs over ${\mathbb Z}_{v}$ when $i$ and $j$ run over the integer intervals
 $[0,q^2]$ and $[1,q]$ respectively.
Then $\mathcal{S}=\{P_i^j: 0\leq i\leq q^2, 0\leq j\leq q\}$ is a set of $(q^4-1)/(q-1)$ spreads. It is readily checked that $\mathcal{S}$ satisfies the three conditions in Definition \ref{e}: $(1)$ for each line $L_i\in P$, one can check that $L_i\in P_i^j$ for each $0\leq j\leq q$; (2) for each line $L\not\in P$, $L$ occurs in exactly one full orbit $\mathcal{O}$ of $\pg(3,q)$; since $P_0^0$ contains $q$ lines from $\mathcal{O}$, there are exactly $q$ spreads in $\mathcal{S}$ containing $L$; (3) each spread in $\mathcal{S}$ cannot contain two lines of $P$ at the same time.
\end{proof}

\begin{Example}\label{q2}
$\pg(3,2)$ has Property $\e$.
\end{Example}

\begin{proof} Although Meszka \cite[Lemma 2]{m13} provided an example of $\pg(3,2)$ having Property $\e$, we here provide a different construction by using Lemma \ref{withE}. Let $\f_{2^4}=\f_2[x]/\langle x^4+x+1\rangle$ where $x^4+x+1$ is an irreducible polynomial over $\f_2$ and $x\in \f_2[x]$ is a fixed primitive element of $\f_{2^4}$. Each point in $\pg(3,2)$ can be represented by $x^i$ or simply $i$ for some $0\leq i\leq 14$. By Lemma \ref{orbit}, $\pg(3,2)$ has a unique short orbit
$$P=\{\{0,5,10\}+i: i\in \{0,1,2,3,4\}\},$$
and two full orbits $\mathcal{O}_1=\{\{0,11,12\}+i: i\in\mathbb{Z}_{15}\}$ and $\mathcal{O}_2=\{\{0,2,8\}+i: i\in\mathbb{Z}_{15}\}$. We can take two disjoint lines $\{0,11,12\}+1=\{1,12,13\}$ and $\{0,11,12\}+7=\{3,4,7\}$ from $\mathcal{O}_1$ and two disjoint lines $\{0,2,8\}+6=\{6,8,14\}$ and $\{0,2,8\}+9=\{2,9,11\}$ from $\mathcal{O}_2$ such that
$$P_0^0=\{\{1,12,13\},\{3,4,7\},\{6,8,14\},\{2,9,11\},\{0,5,10\}\}$$
forms a spread satisfying the condition in Lemma \ref{withE}. Therefore, $$\mathcal{S}=\{P_i^j=P_0^0+5j+i: 0\leq i\leq 4, 0\leq j\leq 2\}$$
is a set of $15$ spreads which ensure $\pg(3,2)$ has Property $\e$. We list all $P_i^j$ explicitly as follows to facilitate the reader to check it.
\begin{center}
\begin{tabular}{c}
$P_0^0=\{\underline{\{0,5,10\}},\{1,12,13\},\{3,4,7\},\{6,8,14\},\{2,9,11\}\}$.\\
$P_0^1=\{\underline{\{0,5,10\}},\{2,3,6\},\{8,9,12\},\{4,11,13\},\{1,7,14\}\}$.\\
$P_0^2=\{\underline{\{0,5,10\}},\{7,8,11\},\{2,13,14\},\{1,3,9\},\{4,6,12\}\}$.\\
$P_1^0=\{\underline{\{1,6,11\}},\{2,13,14\},\{4,5,8\},\{0,7,9\},\{3,10,12\}\}$.\\
$P_1^1=\{\underline{\{1,6,11\}},\{3,4,7\},\{9,10,13\},\{5,12,14\},\{0,2,8\}\}$.\\
$P_1^2=\{\underline{\{1,6,11\}},\{8,9,12\},\{0,3,14\},\{2,4,10\},\{5,7,13\}\}$.\\
$P_2^0=\{\underline{\{2,7,12\}},\{0,3,14\},\{5,6,9\},\{1,8,10\},\{4,11,13\}\}$.\\
$P_2^1=\{\underline{\{2,7,12\}},\{4,5,8\},\{10,11,14\},\{0,6,13\},\{1,3,9\}\}$.\\
$P_2^2=\{\underline{\{2,7,12\}},\{9,10,13\},\{0,1,4\},\{3,5,11\},\{6,8,14\}\}$.\\
$P_3^0=\{\underline{\{3,8,13\}},\{0,1,4\},\{6,7,10\},\{2,9,11\},\{5,12,14\}\}$.\\
$P_3^1=\{\underline{\{3,8,13\}},\{5,6,9\},\{0,11,12\},\{1,7,14\},\{2,4,10\}\}$.\\
$P_3^2=\{\underline{\{3,8,13\}},\{10,11,14\},\{1,2,5\},\{4,6,12\},\{0,7,9\}\}$.\\
$P_4^0=\{\underline{\{4,9,14\}},\{1,2,5\},\{7,8,11\},\{3,10,12\},\{0,6,13\}\}$.\\
$P_4^1=\{\underline{\{4,9,14\}},\{6,7,10\},\{1,12,13\},\{0,2,8\},\{3,5,11\}\}$.\\
$P_4^2=\{\underline{\{4,9,14\}},\{0,11,12\},\{2,3,6\},\{5,7,13\},\{1,8,10\}\}$.\\
\end{tabular}
\end{center}
The underlined lines are from the spread $P$.

In this example, in order to find $P_0^0$ one can, without loss of generality, choose the line $\{0,5,10\}$ first and then search for the other four lines. A way to speed up the search is the use of Frobenius maps of $\f_{2^4}$ over $\f_2$. It is readily checked that $\{\{1,12,13\}\times 2^l \text{ (mod } 15):0\leq l\leq 3\}=\{\{1,12,13\},\{2,9,11\},\{4,3,7\},\{8,6,14\}\}=P_0^0\setminus\{\{0,5,10\}\}.$
Thus it suffices to find the initial line $\{1,11,12\}$ and then develop it through multiplying by all powers of $2$ to obtain all the other lines except for the fixed one $\{0,5,10\}$.
\end{proof}

\begin{Lemma}\label{q4}
$\pg(3,4)$ has Property $\e$.
\end{Lemma}

\begin{proof} Let $\f_{2^8}=\f_2[x]/\langle x^8+x^6+x^3+x^2+1\rangle$ where $x^8+x^6+x^3+x^2+1$ is an irreducible polynomial over $\f_2$ and $x\in \f_2[x]$ is a fixed primitive element of $\f_{2^8}\cong\f_{4^4}$. Each point in $\pg(3,4)$ can be represented by $x^i$ or simply $i$ for some $0\leq i\leq 84$. By Lemma \ref{orbit}, $\pg(3,4)$ has a unique short orbit $P=\{\{0,17,34,51,68\}+i: 0\leq i\leq 16\}$ and four full orbits
\begin{center}
\begin{tabular}{ll}
$\mathcal{O}_1=\{\{0,1,13,23,70\}+i: i\in\mathbb{Z}_{85}\}$, & $\mathcal{O}_2=\{\{0,2,26,46,55\}+i: i\in\mathbb{Z}_{85}\}$,\\
$\mathcal{O}_3=\{\{0,3,21,48,81\}+i: i\in\mathbb{Z}_{85}\}$, & $\mathcal{O}_4=\{\{0,5,36,71,79\}+i: i\in\mathbb{Z}_{85}\}$.\\
\end{tabular}
\end{center}
We take two initial lines $\{2,12,59,74,75\}$ and $\{14,29,30,42,52\}$, and then develop them through multiplying by $2^l$, $0\leq l\leq 7$, to obtain $16$ lines such that each full orbit contains exactly $4$ lines. Specifically the $16$ lines are
\begin{center}
\begin{tabular}{l}
$\mathcal{D}_1=\{\{0,1,13,23,70\}+j: j\in\{74,29,9,54\}\}\subseteq \mathcal{O}_1$,\\
$\mathcal{D}_2=\{\{0,2,26,46,55\}+j: j\in\{63,58,18,23\}\}\subseteq \mathcal{O}_2$,\\
$\mathcal{D}_3=\{\{0,3,21,48,81\}+j: j\in\{45,35,40,50\}\}\subseteq \mathcal{O}_3$,\\
$\mathcal{D}_4=\{\{0,5,36,71,79\}+j: j\in\{11,76,1,21\}\}\subseteq \mathcal{O}_4$.
\end{tabular}
\end{center}
Then $\{\{0,17,34,51,68\}\}\cup (\bigcup_{i=1}^4\mathcal{D}_i)$ forms a spread $P_0^0$ satisfying the condition in Lemma \ref{withE}. Therefore, $\pg(3,4)$ has Property $\e$.
\end{proof}

\begin{Lemma}\label{q8}
$\pg(3,8)$ has Property $\e$.
\end{Lemma}

\begin{proof}
Let $\f_{2^{12}}=\f_2[x]/\langle x^{12}+x^{10}+x^2+x+1\rangle$ where $x^{12}+x^{10}+x^2+x+1$ is an irreducible polynomial over $\f_2$ and $x\in \f_2[x]$ is a fixed primitive element of $\f_{2^{12}}\cong\f_{8^4}$. Each point in $\pg(3,8)$ can be represented by $x^i$ or simply $i$ for some $0\leq i\leq 584$. By Lemma \ref{orbit}, $\pg(3,8)$ has a unique short orbit $P=\{\{0,65,130,195,260,325,390,455,520\}+i: 0\leq i\leq 64\}$ and 8 full orbits
\begin{center}
\begin{tabular}{l}
$\mathcal{O}_1=\{\{0, 1, 37, 84, 89, 328, 404, 490, 558\}+i: i\in\mathbb{Z}_{585}\}$,\\
$\mathcal{O}_2=\{\{0,2,71,74,168,178,223,395,531\}+i: i\in\mathbb{Z}_{585}\}$,\\ $\mathcal{O}_3=\{\{0,4,142,148,205,336,356,446,477\}+i: i\in\mathbb{Z}_{585}\}$,\\
$\mathcal{O}_4=\{\{0,7,16,29,153,174,235,254,568\}+i: i\in\mathbb{Z}_{585}\}$,\\
$\mathcal{O}_5=\{\{0,8,87,127,284,296,307,369,410\}+i: i\in\mathbb{Z}_{585}\}$,\\
$\mathcal{O}_6=\{\{0,14,32,58,306,348,470,508,551\}+i: i\in\mathbb{Z}_{585}\}$,\\ $\mathcal{O}_7=\{\{0,15,155,204,255,354,389,414,467\}+i: i\in\mathbb{Z}_{585}\}$,\\
$\mathcal{O}_8=\{\{0,30,123,193,243,310,349,408,510\}+i: i\in\mathbb{Z}_{585}\}$,\\
\end{tabular}
\end{center}
We take six initial lines:
\begin{center}
\begin{tabular}{ll}
$\{1,77,163,231,258,259,295,342,347\}$,&
$\{7,12,251,327,413,481,508,509,545\}$,\\
$\{15,42,43,79,126,131,370,446,532\}$,&
$\{26,53,54,90,137,142,381,457,543\}$,\\
$\{22,121,156,181,234,352,367,507,556\}$,&
$\{18,136,151,291,340,391,490,525,550\}$,
\end{tabular}
\end{center}
and then develop them through multiplying by $2^l$, $0\leq l\leq 11$, to obtain $64$ lines such that each full orbit contains exactly $8$ lines (note that $\{22, 121, 156, 181, 234, 352, 367, 507, 556\}$ can only
generate 4 distinct lines). Specifically the 64 lines are
\begin{center}
\begin{tabular}{l}
$\mathcal{D}_1=\{\{0, 1, 37, 84, 89, 328, 404, 490, 558\}+j:j\in\{258, 159, 508, 364, 42, 375, 53, 494\}\}$,\\
$\mathcal{D}_2=\{\{0,2,71,74,168,178,223,395,531\}+j: j\in\{516, 318, 431, 143, 84, 165, 106, 403\}\}$,\\
$\mathcal{D}_3=\{\{0,4,142,148,205,336,356,446,477\}+j: j\in\{447, 51, 277, 286, 168, 330, 212, 221\}\}$,\\
$\mathcal{D}_4=\{\{0,7,16,29,153,174,235,254,568\}+j: j\in\{33, 204, 523, 559, 87, 150, 263, 299\}\}$,\\
$\mathcal{D}_5=\{\{0,8,87,127,284,296,307,369,410\}+j: j\in\{309, 102, 554, 572, 336, 75, 424, 442\}\}$,\\
$\mathcal{D}_6=\{\{0,14,32,58,306,348,470,508,551\}+j: j\in\{66, 408, 461, 533, 174, 300, 526, 13\}\}$,\\
$\mathcal{D}_7=\{\{0,15,155,204,255,354,389,414,467\}+j: j\in\{46, 100, 136, 190, 352, 406, 415, 469\}\}$,\\
$\mathcal{D}_8=\{\{0,30,123,193,243,310,349,408,510\}+j: j\in\{92, 119, 200, 227, 245, 272, 353, 380\}\}$.
\end{tabular}
\end{center}
Then $\{\{0,65,130,195,260,325,390,455,520\}\}\cup (\bigcup_{i=1}^8\mathcal{D}_i)$ forms a spread $P_0^0$ satisfying the condition in Lemma \ref{withE}. Therefore, $\pg(3,8)$ has Property $\e$.
\end{proof}

\begin{Lemma}\label{q16}
$\pg(3,16)$ has Property $\e$.
\end{Lemma}

\begin{proof}
Let $\f_{2^{16}}=\f_2[x]/\langle x^{16}+x^{14}+x^{12}+x^7+x^6+x^4+x^2+x+1\rangle$ where $x^{16}+x^{14}+x^{12}+x^7+x^6+x^4+x^2+x+1$ is an irreducible polynomial over $\f_2$ and $x\in \f_2[x]$ is a fixed primitive element of $\f_{2^{12}}\cong\f_{{16}^4}$. Each point in $\pg(3,16)$ can be represented by $x^i$ or simply $i$ for some $0\leq i\leq 4368$. By Lemma \ref{orbit}, $\pg(3,16)$ has a unique short orbit and 16 full orbits. We take 16 initial lines:
\begin{center}
\begin{tabular}{r}
$\{0,1,727,733,997,1133,1877,1917,2477,2864,3070,3172,3453,3594,3672,3928,4279\}+i$\\
$\pmod{4369}, {\text{  where    }}i \in \{117, 135, 550, 660, 782, 2563, 3254, 3529\}$;\\
$\{0,7,898,1060,1146,1362,1392,1817,2054,2287,2323,2504,2875,3723,3749,4205,4227\}+i$\\
$\pmod{4369}, {\text{  where    }}i \in \{183, 818, 1495, 1554, 1754, 1981, 2088, 3979\}$,
\end{tabular}
\end{center}
and then develop them through multiplying by $2^l$, $0\leq l\leq 15$, to obtain $256$ lines such that each full orbit contains exactly $16$ lines. The 256 lines together the line $\{257 j: 0\leq j\leq 16\}$ form a spread $P_0^0$ satisfying the condition in Lemma \ref{withE}. Therefore, $\pg(3,16)$ has Property $\e$.
\end{proof}

It only took 4 seconds by a common personal computer to search for Property $\e$ of $\pg(3,16)$.
We believe that one can use computer and apply Lemma \ref{withE} together with the use of Frobenius maps of $\f_{2^{4m}}$ over $\f_2$ to confirm Property $\e$ of $\pg(3,2^m)$ for bigger $m$. It is reasonable to conjecture that $\pg(3,2^m)$ has Property $\e$ for any positive integer $m$.

On the other hand, by exhaustive search, we found that Lemma \ref{withE} cannot be applied for $\pg(3,3)$ and $\pg(3,5)$. To deal with the case of $\pg(3,3)$, we searched for $40$ spreads satisfying the conditions in Definition \ref{e} without use of any group action.

\begin{Lemma} \label{q3}
$\pg(3,3)$ has Property $\e$.
\end{Lemma}

\begin{proof} Let $\f_{3^4}=\f_3[x]/\langle x^4+x+2\rangle$ where $x^4+x+2$ is an irreducible polynomial over $\f_3$ and $x\in \f_3[x]$ is a fixed primitive element of $\f_{3^4}$. Each point in $\pg(3,3)$ can be represented by $x^i$ or simply $i$ for some $0\leq i\leq 39$. By Lemma \ref{orbit}, the 130 lines of $\pg(3,3)$ can be given below according to orbits:
\begin{center}
\begin{tabular}{lll}
$L_i=\{0,1,4,13\}+i\pmod{40}$, & $L_{40+i}=\{0,2,17,24\}+i\pmod{40}$, \\ $L_{80+i}=\{0,5,26,34\}+i\pmod{40}$,&
$L_{120+j}=\{0,10,20,30\}+j\pmod{40}$,
\end{tabular}
\end{center}
where $i\in\mathbb{Z}_{40}$ and $0\leq j\leq 9$. We take the spread
$$P=\{L_0, L_2, L_{32}, L_{127}, L_8, L_{10}, L_{110}, L_{18}, L_{114}, L_{25}\}$$
as the special spread in Definition \ref{e}, and a set $\mathcal{S}$ of $40$ spreads of $\pg(3,3)$ satisfying the conditions in Definition \ref{e} is listed below. To save space we write $i$ instead of the line $L_i$ for $0\leq i\leq 129$.
\begin{center}
\begin{tabular}{lll}
$\{\underline{0},6,11,21,29,35,43,54,103,112\}$,&$\{\underline{0},6,11,21,29,35,43,54,103,112\}$,\\
$\{\underline{0},6,11,21,29,35,43,54,103,112\}$,&$\{\underline{0},5,20,38,66,70,75,102,109,111\}$,\\
$\{\underline{2},4,9,19,24,26,34,52,56,121\}$,&$\{\underline{2},4,9,19,24,26,34,52,56,121\}$,\\
$\{\underline{2},4,9,19,24,26,34,52,56,121\}$,&$\{\underline{2},7,12,17,22,28,36,85,113,124\}$,\\
$\{\underline{32},3,46,51,57,65,77,78,104,106\}$,&$\{\underline{32},3,46,51,57,65,77,78,104,106\}$,\\
$\{\underline{32},3,46,51,57,65,77,78,104,106\}$,&$\{\underline{32},7,12,14,22,30,37,40,44,129\}$,\\
$\{\underline{127},5,12,22,28,38,59,71,120,124\}$,&$\{\underline{127},39,58,61,64,69,72,90,105,108\}$,\\
$\{\underline{127},39,58,61,64,69,72,90,105,108\}$,&$\{\underline{127},39,58,61,64,69,72,90,105,108\}$,\\
$\{\underline{8},14,28,30,42,62,63,73,76,91\}$,&$\{\underline{8},14,16,31,33,38,63,66,99,116\}$,\\
$\{\underline{8},15,41,62,67,70,73,76,80,117\}$,&$\{\underline{8},13,15,45,74,81,117,118,119,120\}$,\\
$\{\underline{10},15,20,80,81,83,84,93,111,122\}$,&$\{\underline{10},30,40,60,87,107,125,126,128,129\}$,\\
$\{\underline{10},16,47,53,59,80,81,119,122,128\}$,&$\{\underline{10},33,42,68,93,95,102,111,115,118\}$,\\
$\{\underline{110},41,44,45,49,50,55,76,94,117\}$,&$\{\underline{110},36,42,85,87,97,100,107,109,128\}$,\\
$\{\underline{110},1,27,48,60,74,92,101,123,129\}$,&$\{\underline{110},1,27,48,49,59,60,92,93,109\}$,\\
$\{\underline{18},23,55,66,84,87,88,91,100,115\}$,&$\{\underline{18},49,79,82,83,92,107,120,124,125\}$,\\
$\{\underline{18},23,37,68,86,88,89,101,119,16\}$,&$\{\underline{18},73,36,67,68,79,88,100,101,118\}$,\\
$\{\underline{114},13,47,83,86,96,98,113,116,125\}$,&$\{\underline{114},13,23,41,45,71,75,84,86,96\}$,\\
$\{\underline{114},17,37,63,67,71,89,99,122,126\}$,&$\{\underline{114},1,47,48,53,97,98,113,115,126\}$,\\
$\{\underline{25},17,31,50,62,82,91,94,95,123\}$,&$\{\underline{25},5,20,31,50,53,79,82,94,97\}$,\\
$\{\underline{25},40,44,70,74,75,85,95,102,123\}$,&$\{\underline{25},7,27,33,55,89,96,98,99,116\}$,
\end{tabular}
\end{center}
where  the underlined numbers correspond to the lines from the spread $P$.
\end{proof}

%To shorten this Section, we remove the conclusion to the Section 4. In Section 4,  based on algebraic methods and computer-assisted search, we give an example [see Example \ref{q2}] that is different from Meszka and  show that $\pg(3,q)$ has Property $\e$ for $q\in\{3,4,8\}$[see Lemmas \ref{q4}, \ref{q8} and \ref{q3}].

\section{The chromatic index of PG$(n,q)$}\label{sec:3}

\begin{Definition}
Given a subspace $\pg(n-2,q)$ of $\pg(n,q)$, let $\ipg(n,q;n-2)$ denote the geometry that is obtained from $\pg(n,q)$ by removing all lines of the given $\pg(n-2,q)$.
\end{Definition}

\begin{Lemma}\label{lem3_33}$\ipg(3,q;1)$ is $(q^2+q+1)$-colorable, where $q(q+1)$ colors are used to color all lines incident with the given $\pg(1,q)$.
\end{Lemma}

\begin{proof}
By Theorem \ref{thm:known results}, $\chi'(\pg(3,q))=q^2+q+1$, so the lines of $\pg(3,q)$ can be partitioned into $q^2+q+1$ $\sd$s. There is exactly one $\sd$ $P$ containing the line of the given $\pg(1,q)$, and the remaining lines in $P$ are not incident with the points of $\pg(1,q)$. Thus $q(q+1)$ colors are used to color all lines incident with the points of $\pg(1,q)$.
\end{proof}

A {\em group divisible design} (GDD) $k$-GDD is a triple ($X, {\cal G},{\cal A}$) satisfying that: ($1$) $\cal G$ is a partition of a finite set $X$ into subsets (called {\em groups}); ($2$) $\cal A$ is a set of $k$-subsets of $X$ (called {\em blocks}), such that every $2$-subset of $X$ is either contained in exactly one block or in exactly one group, but not in both. If $\cal G$ contains $u_i$ groups of size $g_i$ for $1\leq i\leq r$, then $g_1^{u_1}g_2^{u_2}\cdots g_r^{u_r}$ is called the {\em type} of the GDD. A $k$-GDD of type $n^k$ is often called a {\em transversal design} and denoted by a TD$(k,n)$. It is known that a TD$(q+1, q)$ exists for any prime power $q$ (cf. \cite{bjl99}).

The following theorem is a generalization of \cite[Lemma 5]{m13} which only deals with the case $q=2$.

\begin{Theorem}\label{thm:idea}	
\begin{itemize}
\item[{\rm(1)}] Suppose that $\pg(3,q)$ has Property $\e$. Then $\chi'(\pg(n,q))=\gaussm{n}{1}{q}$ for any odd integer $n\geq 5$.
\item[{\rm(2)}] Suppose that $\pg(3,q)$ has Property $\e$ and  $\chi'(\pg(4,q))=\gaussm{4}{1}{q}+q+1$. If $\ipg(4,q;2)$ is $(\gaussm{4}{1}{q}+q+1)$-colorable, where $q^2(q+1)$ colors are used to color all lines incident with the given $\pg(2,q)$, then $\chi'(\pg(n,q))=\gaussm{n}{1}{q}+q+1$ for any even integer $n\geq 6$.
\end{itemize}
\end{Theorem}

\begin{proof} Let $\f_q^*=\langle\alpha\rangle$ where $\alpha$ is a primitive element in $\f_q$.
Take two linearly independent elements $x$ and $y$ in $\f_{q^2}$ over $\f_q$ and let $\langle x,y\rangle=\{x,y,x+y,x+\alpha y,\ldots,x+\alpha^{q-2}y\}$ be the unique line in $\pg(1,q)$ defined on $\f_{q^2}$.

\vspace{0.2cm}

\underline{First we give a construction for $\pg(n,q)$.} This construction is standard from the point of view of combinatorial design theory. Roughly speaking, we take a $(q+1)$-GDD of type $(q+1)^{q^2+1}$, and then give every point weight $q^{n-3}$ to get a $(q+1)$-GDD of type $(q^{n-3}(q+1))^{q^2+1}$. Fill in each group by using an $S(2,q+1,\gaussm{n-1}{1}{q})$ containing an $S(2,q+1,\gaussm{n-3}{1}{q})$ as a subdesign to obtain an $S(2,q+1,\gaussm{n+1}{1}{q})$, which happens to be our desired $\pg(n,q)$. The details are as follows.

Let $(X_0,\A)$ be an $S(2,q+1,\gaussm{4}{1}{q})$ formed by the points and lines of $\pg(3,q)$, where $X_0=\{(v_1,v_2): v_1\in \langle x,y\rangle, v_2\in \f_{q^2}\}\cup\{(0,v'_2): v'_2\in \langle x,y\rangle\}$ (note that $0\in \f_{q^2}$). Since $\pg(3,q)$ has a spread, assume that $\mathcal{G}_0=\{G_1,G_2,\ldots,G_{q^2+1}\}$ is a spread of $\pg(3,q)$. Write $\mathcal{A}_0=\A\setminus \mathcal{G}_0$. Then $(X_0,\mathcal{G}_0,\mathcal{A}_0)$ is a $(q+1)$-GDD of type $(q+1)^{q^2+1}$.

For each line $L_{a,b}=\{a,b,a+b,a+\alpha b,\ldots,a+\alpha^{q-2}b\}\in \mathcal{A}_0$, construct a TD$(q+1,q^{n-3})$ $(X_{L_{a,b}},\mathcal{G}_{L_{a,b}},\A_{L_{a,b}})$ where $X_{L_{a,b}}=\{(u_1,u_2): u_1\in L_{a,b}, u_2\in \f_{q^{n-3}}\}$, $\mathcal{G}_{L_{a,b}}=\{\{(u_1,u_2): u_2\in \mathbb{F}_{q^{n-3}}\}: u_1\in L_{a,b}\}$ and $\A_{L_{a,b}}=\{\{(a,u_2),(b,u'_2),(a+b,u_2+u'_2),\ldots,(a+\alpha^{q-2}b,u_2+\alpha^{q-2}u'_2)\}: u_2,u'_2\in \mathbb{F}_{q^{n-3}}\}$. Note that each $u_1\in L_{a,b}$ is either of the form $(v_1,v_2)$ where $v_1\in \langle x,y\rangle$ and $ v_2\in \f_{q^2}$, or of the form $(0,v'_2)$ where $v'_2\in \langle x,y\rangle$, and consequently each block in $\A_{L_{a,b}}$ can be seen as a line of $\pg(n,q)$ whose underlying $(n+1)$-dimensional vector space $\vn(n+1,q)$ is defined on $\f_{q^2}\oplus\f_{q^2}\oplus\f_{q^{n-3}}$.

It follows that we obtain a $(q+1)$-GDD of type $(q^{n-3}(q+1))^{q^2+1}$, $(X_1,\mathcal{G}_1,\mathcal{A}_1)$, where $X_1=\{(v_1,v_2,u_2): v_1\in \langle x,y\rangle, v_2\in \f_{q^2}, u_2\in \f_{q^{n-3}}\}\cup\{(0,v'_2,u_2): v'_2\in \langle x,y\rangle, u_2\in \f_{q^{n-3}}\}$, $\mathcal{G}_1=\{G_i'=\{(w_1,w_2,u_2): (w_1,w_2)\in G_i, u_2\in \f_{q^{n-3}}\}: 1\leq i\leq q^2+1\}$ and $\mathcal{A}_1=\bigcup_{L_{a,b}\in \mathcal{A}_0} \A_{L_{a,b}}$.

Construct $\pg(n-4,q)$ such that its underlying $(n-3)$-dimensional vector space $\vn(n-3,q)$ is defined on $\{(0,0,u_2): u_2\in \f_{q^{n-3}}\}$ (note that $0\in \f_{q^2}$) . Let $Y$ and $\cal C$ be the sets of points and lines of the $\pg(n-4,q)$, respectively. Then $(Y,{\cal C})$ forms an $S(2,q+1,\gaussm{n-3}{1}{q})$.

For each $1\leq i\leq q^2+1$, construct an $\ipg(n-2,q;n-4)$ with the set of points $G_i'\cup Y$ such that $Y$ is the set of points of its underlying $\pg(n-4,q)$. Denote by $\B_i$ the set of lines of the $\ipg(n-2,q;n-4)$. Then $\B_i\cup {\cal C}$ forms an $S(2,q+1,\gaussm{n-1}{1}{q})$ that contains an $S(2,q+1,\gaussm{n-3}{1}{q})$ as a subdesign.

Let $X_2=X_1\cup Y$ and $\A_2=\A_1\cup{\cal C}\cup(\bigcup_{i=1}^{q^2+1}\B_i)$. Then $X_2$ and $\A_2$ form the sets of points and lines of $\pg(n,q)$, respectively. That is, $(X_2,\A_2)$ is an $S(2,q+1,\gaussm{n+1}{1}{q})$.

\vspace{0.2cm}

\underline{Next we color the lines of $\pg(n,q)$.} Let
$$
    c(n,q)=\left\{
          \begin{array}{ll}
               \gaussm{n}{1}{q}, & \text{if}\ n\ \text{is odd}; \vspace{0.2cm}\\
               \gaussm{n}{1}{q}+q+1, & \text{if}\ n\ \text{is even}.
          \end{array}
       \right.
$$
We introduce the notion of property R of $\ipg(n,q;n-2)$. Recall that $\ipg(n,q;n-2)$ is the geometry that is obtained by removing all lines of a given subspace $\pg(n-2,q)$ from $\pg(n,q)$. If $\ipg(n,q;n-2)$ is $c(n,q)$-colorable such that $q^{n-2}(q+1)$ colors are used to color all lines incident with points of the given $\pg(n-2,q)$, then $\ipg(n,q;n-2)$ is said to be with {\em property R}. By Lemma \ref{lem3_33}, $\ipg(3,q;1)$ has property R, and by assumption, $\ipg(4,q;2)$ has property R.

To color the lines of $\pg(n,q)$, we proceed by induction on $n$. By Theorem \ref{thm:known results}, $\chi'(\pg(3,q))=c(3,q)$, and by assumption $\chi'(\pg(4,q))=c(4,q)$. Let $n\geq 5$. Assume that $\chi'(\pg(n-2,q))=c(n-2,q)$ and $\ipg(n-2,q;n-4)$ is $c(n-2,q)$-colorable such that $q^{n-4}(q+1)$ colors are used to color all lines incident with the points of the underlying $\pg(n-4,q)$. We shall prove that $\chi'(\pg(n,q))=c(n,q)$.

We color the lines in $\A_2$ by the following three steps.

Step 1. We color the lines in $\B_1 \cup {\cal C}$, which forms $\pg(n-2,q)$ with the set of points $G'_1\cup Y$. By the induction hypothesis, $\chi'(\pg(n-2,q))=c(n-2,q)$, so $\pg(n-2,q)$ can be colored using a set of $c(n-2,q)$ colors. We split this color set into two disjoint subsets $C_1$ and $C^*$ of cardinalities $c_1=q^{n-4}(q+1)$ and $c_2=c(n-2,q)-c_1$, respectively.

Step 2. We color the lines in $\bigcup_{i=2}^{q^2+1}\B_i$. For each $2\leq i\leq q^2+1$, $\B_i$ forms $\ipg(n-2,q;n-4)$ with the set of points $G'_i\cup Y$. By the induction hypothesis, $\B_i$ is $c(n-2,q)$-colorable in such a way that $c_1=q^{n-4}(q+1)$ colors are used to color all lines incident with the points of $\pg(n-4,q)$ which has the point set $Y$. We take a set $C_i$ of $c_1$ colors, which were not previously used, to color all lines of $\B_i$ incident with $Y$ such that $C_{j_1}\cap C_{j_2}=\emptyset$ for any $1\leq j_1<j_2\leq q^2+1$. To color the remaining lines of $\B_i$, we use $c_2=c(n-2,q)-c_1$ colors that come from the color set $C^*$. Thus the total number of colors is $(q^2+1)\cdot c_1+c_2=c(n,q)$.

Step 3. We color the lines in $\A_1=\bigcup_{L_{a,b}\in \mathcal{A}_0} \A_{L_{a,b}}$ using the colors from $\bigcup_{i=1}^{q^2+1}C_i$. Recall that for each $L_{a,b}\in \mathcal{A}_0$, $\A_{L_{a,b}}$ forms a TD$(q+1,q^{n-3})$.

\begin{Claim}\label{claim:TD}
This $\td$ is resolvable, that is, $\A_{L_{a,b}}$ can be partitioned into parallel classes, each of which contains every point of $X_{L_{a,b}}$ exactly once.
\end{Claim}

\begin{proof} Let $\f_{q^{n+1}}^*=\langle\beta\rangle$ such that  $\alpha=\beta^{\frac{q^{n+1}-1}{q-1}}$. Construct a TD$(q+2,q^{n-3})$ on $X_{L_{a,b}}\cup\{(a+\beta b,u_2): u_2\in \mathbb{F}_{q^{n-3}}\}$ with the set of groups $\mathcal{G}_{L_{a,b}}\cup\{(a+\beta b,u_2): u_2\in \mathbb{F}_{q^{n-3}}\}$ and the set of blocks$\A'_{L_{a,b}}=\{\{(a,u_2),(b,u'_2),(a+b,u_2+u'_2),\ldots,(a+\alpha^{q-2}b,u_2+\alpha^{q-2}u'_2), (a+\beta b,u_2+\beta u'_2)\}: u_2,u'_2\in \mathbb{F}_{q^{n-3}}\}$. We shall show that for every $j\in \mathbb{F}_{q^{n-3}}$,
$\N_j=\{B\setminus{(a+\beta b,j)}: (a+\beta b,j)\in B\in\mathcal{A}'_{L_{a,b}}\}$ is a parallel class of $(X_{L_{a,b}},\mathcal{G}_{L_{a,b}},\mathcal{A}_{L_{a,b}})$.

Clearly $\N_j\subseteq \mathcal{A}_{L_{a,b}}$. Since $|\N_j|=q^{n-3}$, it suffices to show that the blocks in $\N_j$ are pairwise disjoint. Assume that there are two distinct blocks $B_1,B_2\in \N_j$ satisfying $|B_1\cap B_2|\geq 1$, where $\{(a,u_2),(b,u'_2)\}\subset B_1$ and $\{(a,u''_2),(b,u'''_2)\}\subset B_2$. Then $j=u_2+\beta u'_2=u''_2+\beta u'''_2$, and either $u_2=u''_2$ or $u'_2=u'''_2$ or there exists $0\leq k\leq q-2$ such that $u_2+\alpha^k u'_2=u''_2+\alpha^k u'''_2$. Clearly the first two cases are impossible. The last case is also impossible since $\alpha^k\neq\beta$ for any $0\leq k\leq q-2$.
\end{proof}

Therefore, $\A_{L_{a,b}}$ is $q^{n-3}$-colorable. We shall choose $q^{n-3}$ colors from $\bigcup_{i=1}^{q^2+1}C_i$ to color lines in $\A_{L_{a,b}}$ for each $L_{a,b}\in \mathcal{A}_0$.

By assumption, $\pg(3,q)$ has Property E. Without loss of generality, let $P=\mathcal{G}_0=\{G_1,G_2$, $\ldots,G_{q^2+1}\}$ be the special $\sd$ of $\pg(3,q)$ in Definition \ref{e}, and $\mathcal{S}=\{P_{G_i}^j:1\leq j\leq q+1,1\leq i\leq q^2+1\}$ be a set of $(q^4-1)/(q-1)$ $\sd$s satisfying the three conditions in Definition \ref{e}, where for each line $G_i\in P$, there are exactly $q+1$ spreads $P_{G_i}^j$ ($1\leq j\leq q+1$) in $\mathcal{S}$ containing $G_i$.

For each $1\leq i\leq q^2+1$, we split the color set $C_i$ into $q+1$ disjoint subsets $C_i^1,C_i^2,\cdots,C_i^{q+1}$, each of cardinality $q^{n-4}$, and assign each subset $C_i^j$, $1\leq j\leq q+1$, to the spread $P_{G_i}^j$. It follows that each spread in $\mathcal{S}$ owns $q^{n-4}$ colors and any two different spreads in $\mathcal{S}$ own different colors. By Property E of $\pg(3,q)$, every line $L_{a,b}\notin P=\mathcal{G}_0$ (i.e., $L_{a,b}\in \mathcal{A}_0$) of $\pg(3,q)$ occurs in exactly $q$ spreads in $\mathcal{S}$, so the number of colors assigned to $L_{a,b}$ is $q^{n-4}\cdot q=q^{n-3}$. Using the $q^{n-3}$ colors, we can color all of the lines in $\A_{L_{a,b}}$.

\vspace{0.2cm}

\underline{Finally we check the correctness of the coloring.} Take any two different intersecting lines $L_1$ and $L_2$ from $\A_2=\A_1\cup{\cal C}\cup(\bigcup_{i=1}^{q^2+1}\B_i)$. It suffices to show that they have assigned different colors.

If $L_1$ and $L_2$ are both from ${\cal C}\cup(\bigcup_{i=1}^{q^2+1}\B_i)$, then it follows from Steps 1 and 2 that $L_1$ and $L_2$ have different colors.

If $L_1$ and $L_2$ are both from $\A_1=\bigcup_{L_{a,b}\in \mathcal{A}_0} \A_{L_{a,b}}$, then either they come from the same $\A_{L_{a,b}}$, or from two different ones, say $\A_{L_{a,b}}$ and $\A_{L_{a',b'}}$, where $L_{a,b}$ and $L_{a',b'}$ are different lines in $\mathcal{A}_0$. In the former case $L_1$ and $L_2$ are assigned different colors by Step 3. In the latter case, since $L_1$ and $L_2$ are intersecting, $L_{a,b}$ must intersect with $L_{a',b'}$, and so $L_{a,b}$ and $L_{a',b'}$ belong to different spreads in $\cal S$; by Step 3, they have different colors.

If $L_1$ is from $\A_1$ and $L_2$ is from $\cal C$, since $\A_1$ is defined on $X_1$ and $\cal C$ is defined on $Y$, $L_1$ and $L_2$ are disjoint, a contradiction.

If $L_1$ is from $\A_1$ and $L_2$ is from $\bigcup_{i=1}^{q^2+1}\B_i$, then $L_1\in \A_{L_{a,b}}$ for some $L_{a,b}\in \mathcal{A}_0$, and $L_2\in \B_i$ for some $1\leq i\leq q^2+1$. Since $L_1$ and $L_2$ are intersecting, $L_{a,b}$ must intersect with $G_i$, and so $L_{a,b}$ and $G_i$ belong to different spreads in $\cal S$; by Step 3, they have different colors.

%$L_1\in \A_{L_{a,b}}$ for some $L_{a,b}\in \mathcal{A}_0$. Since $L_1$ and $L_2$ are intersecting, $L_{a,b}$ must intersect with $G_1$, and so $L_{a,b}$ and $G_1$ belong to different spreads in $\cal S$;  by Step 3 they have different colors.

\vspace{0.2cm}

Therefore, combining Lemma \ref{prop:lower bound}, we have $\chi'(\pg(n,q))=c(n,q)$. Note that $\A_1\cup(\bigcup_{i=2}^{q^2+1}\B_i)$ forms $\ipg(n,q;n-2)$, and the above coloring procedure implies that this $\ipg(n,q;n-2)$ is with property R, that is, this $\ipg(n,q;n-2)$ is $c(n,q)$-colorable such that $q^{n-2}(q+1)$ colors from $\bigcup_{i=2}^{q^2+1} C_i$ are used to color all lines incident with points of the $\pg(n-2,q)$ which has the set of points $G'_1\cup Y$ and the set of lines  $\B_1\cup{\cal C}$.
\end{proof}

\begin{algorithm}[tbht]\setstretch{1.3}
\caption{Construction for $\pg(n,q)$ with $n\geq 5$}\label{suan1}
\KwIn{Let $\alpha$ be a primitive element in $\f_q$.}

\SetKwInOut{KIN}{Step $1$}
\KIN{Construct $\pg(3,q)$ on $\f_{q^{4}}\times\{0\}$ (note that $0\in \f_{q^{n-3}}$) with the set $\A$ of lines. Let $\mathcal{G}_0=\{G_1,G_2,\ldots,G_{q^2+1}\}$ be a spread of $\pg(3,q)$. Write $\mathcal{A}_0=\A\setminus \mathcal{G}_0$.}

\SetKwInOut{KIN}{Step $2$}
\KIN{For each line $L_{a,b}=\{a,b,a+b,a+\alpha b,\ldots,a+\alpha^{q-2}b\}\in \mathcal{A}_0$, let $$\A_{L_{a,b}}=\{\{(a,u_2),(b,u'_2),(a+b,u_2+u'_2),\ldots,(a+\alpha^{q-2}b,u_2+\alpha^{q-2}u'_2)\}: u_2,u'_2\in \mathbb{F}_{q^{n-3}}\}$$ and $\mathcal{A}_1=\bigcup_{L_{a,b}\in \mathcal{A}_0} \A_{L_{a,b}}$.}

\SetKwInOut{KIN}{Step $3$}
\KIN{Construct $\pg(n-4,q)$ on $\{0\}\times\f_{q^{n-3}}$ (note that $0\in \f_{q^4}$) with the set $\cal C$ of lines.}

\SetKwInOut{KIN}{Step $4$}
\KIN{For each $1\leq i\leq q^2+1$, construct an $\ipg(n-2,q;n-4)$ on $\langle G_i\rangle\times \f_{q^{n-3}}$ such that the underlying $\pg(n-4,q)$ is constructed on $\{0\}\times\f_{q^{n-3}}$. Denote by $\B_i$ the set of lines of the $\ipg(n-2,q;n-4)$.}

\KwOut{$\A_2=\A_1\cup{\cal C}\cup(\bigcup_{i=1}^{q^2+1}\B_i)$ forms the set of lines of $\pg(n,q)$, which is constructed on $\f_{q^4}\times \f_{q^{n-3}}$.}
\end{algorithm}

\begin{proof}[{\bf Proof of Theorem \ref{thm:main-1}}]
By Lemmas \ref{q4}, \ref{q8}, \ref{q16} and \ref{q3}, $\pg(3,4)$, $\pg(3,8)$, $\pg(3,16)$ and $\pg(3,3)$ have Property E, so applying Theorem \ref{thm:idea}(1), we have $\chi'(\pg(n,q))=\gaussm{n}{1}{q}$ for any odd integer $n\geq 5$ and $q\in\{3,4,8,16\}$. The case of $n=1$ is trivial. For $n=3$, by Theorem \ref{thm:known results}(2), $\chi'(\pg(3,q))=\gaussm{3}{1}{q}$ for any prime power $q$.
\end{proof}

Finally we give a proof of Theorem \ref{thm:known result-Meszka} by employing Theorem \ref{thm:idea}. It is known that $\chi'(\pg(4,2))=\gaussm{4}{1}{2}+3=18$ (see \cite[Example 2]{m13}) and $\ipg(4,2;2)$ is $(\gaussm{4}{1}{2}+3)$-colorable where $12$ colors are used to color all lines incident with the given $\pg(2,2)$ (see \cite[Lemma 4(b)]{m13}). Since $\pg(3,2)$ has Property $\e$ by Example \ref{q2}, we can apply Theorem \ref{thm:idea} to obtain that $\chi'(\pg(n,2))=\gaussm{n}{1}{2}+3$ for any even integer $n\geq 6$.

\section{Concluding remarks}

Theorems \ref{thm:known results}, \ref{thm:known result-Meszka} and \ref{thm:main-1} motivate us to conjecture that the lower bound for $\chi'(\pg(n,q))$ in Lemma \ref{prop:lower bound} is tight. However, when $n$ is even, it is only known by Theorem \ref{thm:known result-Meszka} that this lower bound for $\chi'(\pg(n,q))$ is tight for $q=2$. It is meaningful to examine whether $\chi'(\pg(4,3))=\gaussm{4}{1}{3}+4=44$ and whether $\ipg(4,3;2)$ is $(\gaussm{4}{1}{3}+4)$-colorable where $36$ colors are used to color all lines incident with the given $\pg(2,q)$. If so, applying Theorem \ref{thm:idea} would yield the exact value of $\chi'(\pg(n,3))$ for any even integer $n\geq 4$. We have no idea how to solve it so far. Actually the points and lines of $\pg(4,3)$ form an $S(2,4,121)$. Very little is known about the chromatic index of block-colorings of an $S(2,4,v)$ with $v\equiv 1\pmod{12}$ (see \cite[Section 10]{rr}).

%In order to apply Theorem \ref{thm:idea} to obtain the exact value of $\chi'(\pg(n,3))$ for any even integer $n\geq 4$, one need to examine whether $\chi'(\pg(4,3))=\gaussm{4}{1}{3}+4=44$. It is not easy to check it.
%Even though we spent much time to do a computer search for it, we haven't had any success so far. It is interesting to determine whether $\chi'(\pg(4,3))=44$.

%It follows from Example \ref{q2} and Lemmas \ref{q4}, \ref{q8}, \ref{q16} and \ref{q3} that $\pg(3,q)$ has Property E for $q\in\{2,3,4,8,16\}$.

Algorithm \ref{suan1} and Algorithm \ref{suan2} are refined from Theorem \ref{thm:idea}. Algorithm \ref{suan1} gives all lines of $\pg(n,q)$. Algorithm \ref{suan2} counts the chromatic index of $\pg(n,q)$ as long as the data in the input area are provided. Take a coloring of $\pg(5,4)$ for example. Start from $\pg(5,4)$ constructed from Algorithm \ref{suan1}. It follows from Theorem \ref{thm:known results},  Lemma \ref{q4} and Lemma \ref{lem3_33} that $\pg(3,4)$ is $21$-colorable, $\pg(3,4)$ has Property $\e$ and $\ipg(3,4;1)$ is $21$-colorable where $20$ colors are used to color all lines incident with the given $\pg(1,4)$. Then apply Algorithm \ref{suan2} to obtain a parallelism of $\pg(5,4)$, which is 341-colorable.

An interesting question is whether $\pg(3,q)$ has Property E for any prime power $q$. By Theorem \ref{thm:idea}, the solution to this problem will give the existence of a parallelism of $\pg(n,q)$ for any odd integer $n$ and any prime power $q$.

\begin{algorithm}[H]\setstretch{1.35}
\caption{Line coloring of $\pg(n,q)$ with $n\geq 5$}\label{suan2}
\KwIn{$\pg(3,q)$ has Property $\e$. $\pg(n-2,q)$ is $c(n-2,q)$-colorable. $\ipg(n-2,q;n-4)$ is $c(n-2,q)$-colorable where $q^{n-4}(q+1)$ colors are used to color all lines incident with the points of the underlying $\pg(n-4,q)$.}

\SetKwInOut{KIN}{Step $1$}
\KIN{Use $c(n-2,q)$ colors to color the lines in $\B_1 \cup {\cal C}$. Split this color set into two disjoint subsets $C_1$ and $C^*$ of sizes $c_1=q^{n-4}(q+1)$ and $c_2=c(n-2,q)-c_1$, respectively.}

\SetKwInOut{KIN}{Step $i$}
\KIN{($2\leq i\leq q^2+1$). Take a set $C_i$ of $c_1$ colors, which were not previously used, to color all lines of $\B_i$ that are incident with $\{0\}\times \f_{q^{n-3}}$ such that $C_{j_1}\cap C_{j_2}=\emptyset$ for any $1\leq j_1<j_2\leq q^2+1$. Use $c_2$ colors that come from the color set $C^*$ to color all the remaining lines of $\B_i$.}
%To color the remaining lines of $\B_i$, use $c_2$ colors that come from the color set $C^*$.

\SetKwInOut{KIN}{Step $q^2+2$}
\KIN{Let $P=\mathcal{G}_0=\{G_1,G_2$, $\ldots,G_{q^2+1}\}$ be the special $\sd$ of $\pg(3,q)$, and $\mathcal{S}=\{P_{G_i}^j:1\leq j\leq q+1,1\leq i\leq q^2+1\}$ be a set of $(q^2+1)(q+1)$ $\sd$s satisfying the three conditions in Definition \ref{e}. For each $1\leq i\leq q^2+1$, we split the color set $C_i$ into $q+1$ disjoint subsets $C_i^1,C_i^2,\cdots,C_i^{q+1}$, each of cardinality $q^{n-4}$, and assign each subset $C_i^j$, $1\leq j\leq q+1$, to the spread $P_{G_i}^j$. Then for each $L_{a,b}\in \mathcal{A}_0$, the number of colors assigned to $L_{a,b}$ is $q^{n-4}\cdot q=q^{n-3}$. Use the $q^{n-3}$ colors to color all of the lines in $\A_{L_{a,b}}$ (apply Claim \ref{claim:TD}).}

\KwOut{$\A_2=\A_1\cup{\cal C}\cup(\bigcup_{i=1}^{q^2+1}\B_i)$ is $c(n,q)$-colorable.}
%{All color classes.}
\end{algorithm}

\subsection*{Acknowledgements}

The authors thank the anonymous referees for their valuable comments and suggestions that helped improve the equality of the paper.

%\section*{Acknowledgments}

%The authors would like to thank Tao Feng from Zhejiang University for helpful discussions.

\end{document}